\documentclass{amsart}

\usepackage{amsmath, amssymb, amsthm}
\usepackage{enumitem}
\usepackage{color}
\usepackage[utf8]{inputenc}

\newtheorem{theorem}{Theorem}[section]
\newtheorem{lemma}[theorem]{Lemma}
\newtheorem{proposition}[theorem]{Proposition}
\newtheorem{corollary}[theorem]{Corollary}

\newtheorem{fact}[theorem]{Fact}
\newtheorem{question}[theorem]{Question}

\newtheorem*{mainthm}{Main Theorem}
\newtheorem*{maincor}{Main Corollary}

\theoremstyle{definition}
\newtheorem{definition}[theorem]{Definition}
\newtheorem{remark}[theorem]{Remark}

\usepackage[pdftex]{hyperref}
\hypersetup{
    colorlinks=true, 
    linktoc=all,     
    linkcolor=blue,  
    allcolors=blue,
}

\newcommand{\cf}{\mathrm{cf}}
\newcommand{\dom}{\mathrm{dom}}
\newcommand{\bb}{\mathbb}

\newcommand{\otp}{\mathrm{otp}}

\newcommand{\zfc}{\mathsf{ZFC}}

\newcommand{\mc}{\mathcal}
\newcommand{\mf}{\mathfrak}
\newcommand{\On}{\mathrm{On}}
\newcommand{\squig}{\rightsquigarrow}

\title{A Galvin-Hajnal theorem for generalized cardinal characteristics}
\author{Chris Lambie-Hanson}
\address{Institute of Mathematics of the Czech Academy of Sciences \\ 
\v{Z}itn\'{a} 25, Praha 1, Czechia}
\email{lambiehanson@math.cas.cz}
\urladdr{http://math.cas.cz/lambiehanson}
\keywords{singular cardinals, cardinal characteristics, cardinal arithmetic, 
Silver's theorem, Galvin-Hajnal theorem}
\subjclass[2010]{03E04, 03E05, 03E17}

\thanks{The results of this paper were presented at the Bar-Ilan University Set 
Theory Seminar in June 2022. We thank the organizers for the invitation and the 
seminar's participants for a number of insightful questions and comments. 
We thank Moti Gitik for valuable conversations regarding Theorem 
\ref{silver_a_tree_thm}. Finally, we thank the anonymous referee for their 
suggestions and corrections.}

\begin{document}

\begin{abstract}
  We prove that a variety of generalized cardinal characteristics, including 
  meeting numbers, the reaping number, and the dominating number, satisfy an 
  analogue of the Galvin-Hajnal theorem, and hence also of Silver's theorem, at 
  singular cardinals of uncountable cofinality.
\end{abstract}

\maketitle

\section{Introduction}

One of the seminal results in cardinal arithmetic, and one of the first indications that 
there are nontrivial $\zfc$ constraints on the behavior of the continuum function at singular 
cardinals, is \emph{Silver's theorem}.

\begin{theorem} [Silver \cite{silver}]
  Suppose that $\kappa$ is a singular cardinal of uncountable cofinality, 
  $\eta < \cf(\kappa)$ is an ordinal, and the set of cardinals
  \[
    \{\mu < \kappa \mid 2^\mu \leq \mu^{+\eta}\}
  \]
  is stationary in $\kappa$. Then $2^\kappa \leq \kappa^{+\eta}$.
\end{theorem}

Silver's original proof of this theorem involves a generic ultrapower argument; a purely 
combinatorial argument for the theorem was soon provided by Baumgartner and Prikry 
\cite{baumgartner_prikry}. Around the same time, a generalization of Silver's theorem 
was proven by Galvin and Hajnal. The following statement of (a corollary of) their 
theorem involves the notion of the \emph{Galvin-Hajnal rank} $\|\varphi\|_S$ of a function 
$\varphi$; see Definition \ref{galvin_hajnal_def} below for its formal definition.

\begin{theorem} [Galvin-Hajnal \cite{galvin_hajnal}] \label{galvin_hajnal_thm}
  Suppose that $\kappa$ is a singular cardinal of uncountable cofinality, 
  $\langle \kappa_i \mid i < \cf(\kappa) \rangle$ is an increasing, continuous 
  sequence of cardinals converging to $\kappa$, $S \subseteq \cf(\kappa)$ is 
  stationary, and $\varphi:S \rightarrow \On$ is a function such that, for all 
  $i \in S$, we have $2^{\kappa_i} \leq \kappa_i^{+\varphi(i)}$. Then 
  $2^\kappa \leq \kappa^{+\|\varphi\|_S}$.
\end{theorem}

This theorem does indeed generalize Silver's theorem, since, as we shall see, 
given any stationary subset $S$ of a regular uncountable cardinal $\theta$, 
and given any ordinal $\eta < \theta$, if $\varphi$ is the constant function on 
$S$ taking value $\eta$, then $\|\varphi\|_S = \eta$.

One of the central aspects of research into cardinal arithmetic is the study 
of certain methods of measuring the ``size" of the power set of a cardinal $\kappa$ 
that are in a sense \emph{finer} than simply looking at the value of $2^\kappa$. 
At singular cardinals, these methods come from two primary sources, with 
some overlap between the two:
\begin{itemize}
  \item Shelah's PCF theory;
  \item generalizations of cardinal characteristics of the continuum to singular cardinals.
\end{itemize}
Certain of these methods are known to satisfy versions of Silver's theorem 
or the Galvin-Hajnal theorem. For example, in \cite[\S 2, Claim 2.4]{cardinal_arithmetic}, 
Shelah proves a variation of Theorem \ref{galvin_hajnal_thm}
involving PCF-theoretic pseudopowers $\mathrm{pp}_J(\kappa)$ and $\mathrm{pp}_J(\kappa_i)$ 
in place of the cardinals $2^\kappa$ and $2^{\kappa_i}$; in 
\cite[Lemma 3.8]{rinot_milner_sauer}, Rinot proves a version of Silver's theorem 
for covering numbers; and in \cite{kojman_density}, Kojman proves that certain \emph{density} numbers satisfy an 
analogue of Silver's theorem (see Section \ref{density_section} for details).

In this paper, we prove versions of the Galvin-Hajnal theorem for a variety of 
cardinal characteristics of the continuum generalized to singular cardinals of uncountable 
cofinality, focusing in particular on meeting numbers, the reaping number, and the dominating number. 
Before proceeding to a summary of our results, let us say a few words about our approach to 
cardinal characteristics at singular cardinals in general. There are often multiple 
natural ways to generalize familiar cardinal characteristics of the continuum to singular cardinals. 
For example, when defining the dominating number $\mathfrak{d}_\kappa$ at a singular cardinal $\kappa$, 
any of the following possible definitions of $\mathfrak{d}_\kappa$ seems potentially reasonable (see 
the end of this introduction for any undefined notation):
\begin{itemize}
  \item $\cf({^\kappa}\kappa, <_0)$, where, given $f,g \in {^\kappa}\kappa$, we let 
  $f <_0 g$ if and only if $|\{i < \kappa \mid g(i) \leq f(i)\}| < \kappa$;
  \item $\cf({^\kappa}\kappa, <_1 )$, where, given $f,g \in {^\kappa}\kappa$, we let $f <_1 g$ if 
  and only if $\{i < \kappa \mid g(i) \leq f(i)\}$ is bounded below $\kappa$;
  \item $\cf({^{\cf(\kappa)}}\kappa, <_2)$, where, given $f,g \in {^{\cf(\kappa)}}\kappa$, 
  we let $f <_2 g$ if and only if $|\{i < \cf(\kappa) \mid g(i) \leq f(i)\}| < \cf(\kappa)$.
\end{itemize}
In all such choices that we face here, we opt for the definition that emphasizes the \emph{cardinality} 
of $\kappa$ over its \emph{cofinality}, as, at least in this context, this seems to be what gives 
rise to the most genuinely new behavior at the singular cardinal $\kappa$. So, for instance, we will define 
$\mathfrak{d}_\kappa$ to be what is called $\cf({^\kappa}\kappa, <_0)$ above. (It is not difficult to 
show that what is called $\cf({^{\cf(\kappa)}}\kappa, <_2)$ above is in fact nothing other than 
$\mathfrak{d}_{\cf(\kappa)}$.)

We also note here that in this paper we are only considering cardinal characteristics at a singular 
cardinal $\kappa$ that are provably strictly greater than $\kappa$. In particular, we are not 
considering the bounding number $\mathfrak{b}_\kappa$, the splitting number $\mathfrak{s}_\kappa$, 
or the almost disjointness number $\mathfrak{a}_\kappa$, since, at least when generalized in accordance 
with the principles laid out in the previous paragraph, these cardinal characteristics are provably 
at most $\mathfrak{b}_{\cf(\kappa)}$, $\mathfrak{s}_{\cf(\kappa)}$, and 
$\mathfrak{a}_{\cf(\kappa)}$, respectively (though we will have more to say about the almost disjointness 
number in Section \ref{question_section}).

A slightly suboptimal but succinct summary of our main results can be stated as 
follows (we refer the reader to Section \ref{canonical_section} for the definition of 
\emph{canonical function} and to Section \ref{instance_section} for the precise definition of 
the cardinal characteristics under consideration):

\begin{maincor}
  Suppose that 
  \begin{itemize}
    \item $\kappa$ is a singular cardinal with $\cf(\kappa) = \theta > \omega$;
    \item $\langle \kappa_i \mid i < \theta \rangle$ is an increasing, continuous 
    sequence of cardinals converging to $\kappa$;
    \item $\beta$ is an ordinal for which the canonical function on $\theta$ of 
    rank $\beta$, $\varphi^\theta_\beta$, exists;
    \item $\mu^\theta \leq \kappa^{+\beta}$ for all $\mu < \kappa$;
    \item $S \subseteq \theta$ is stationary;
    \item $\mf{cc}$ is one of the cardinal characteristics $m(\theta, \kappa)$, $d(\theta, \kappa)$,
    $\mf r_\kappa$, or $\mf d_\kappa$, and, for all $i < \theta$, $\mf{cc}_i$ 
    is the corresponding cardinal characteristic $m(\cf(\kappa_i), \kappa_i)$, $d(\cf(\kappa_i), \kappa_i)$,
    $\mf r_{\kappa_i}$, or $\mf d_{\kappa_i}$;
    \item for all $i \in S$, we have $\mf{cc}_i \leq \kappa_i^{+\varphi^\theta_\beta(i)}$.
  \end{itemize}
  Then $\mf{cc} \leq \kappa^{+\beta}$.
\end{maincor}

The slight suboptimality in this statement comes from the assumption that $\mu^\theta 
\leq \kappa^{+\beta}$ for all $\mu < \kappa$. As we will see, a weaker hypothesis, 
in which $\mu^\theta$ is replaced by some cardinal characteristic that is provably 
at most $\mu^\theta$, is sufficient 
for our results; the precise weakening depends on the specific cardinal characteristic 
under consideration and will require some further notation to state, so we leave the 
exact details for the statement of the Main Theorem at the end of Section 
\ref{instance_section}.

The structure of the remainder of the paper is as follows. In Section 
\ref{canonical_section}, we review the definitions and facts regarding canonical functions 
and the Galvin-Hajnal rank that we will need for our results. In Section 
\ref{density_section}, we recall certain notions of \emph{density}. This is important for 
two reasons: first, because the analogue of Silver's theorem for density numbers 
proven in \cite{kojman_density} was direct inspiration for this paper, and 
secondly and more immediately, these density numbers will appear in the precise 
formulations of our results. After this, we begin with the proof of our main theorem. 
The proofs of our various analogues of the Galvin-Hajnal theorem all have the same 
general shape, so in Section \ref{framework_section} we develop an abstract framework 
that will apply to all of our specific instances. In Section \ref{instance_section}, 
we apply this abstract framework to our cardinal characteristics under consideration 
to obtain our Main Theorem, which is precisely stated at the end of the section. 
Finally, in Section \ref{question_section}, we record 
some questions that remain open and sketch a consistent negative answer to the question 
about whether a version of Silver's theorem holds for the existence of Aronszajn trees at double 
successors of singular cardinals.

\subsection{Notation and conventions:} 
Unless otherwise noted, we believe our notation and terminology to be standard. We refer 
the reader to \cite{jech} for any undefined notions or notations from set theory, and we 
refer the reader to \cite{blass} for an introduction to cardinal characteristics of the 
continuum, generalizations of which form the subject of this paper. 

If $X$ is a set and $\squig$ is a binary relation on 
$X$, then $\cf(X, \squig)$ denotes the minimal cardinality of a subset $Y \subseteq X$ such that, 
for all $x \in X$, there is $y \in Y$ for which $x \squig y$. If $\theta$ is a regular uncountable 
cardinal, then $\mathrm{NS}_\theta$ denotes the nonstationary ideal on $\theta$. If $S \subseteq \theta$ 
is a stationary set, then, formally, $\mathrm{NS}_\theta \restriction S$ is the ideal on $\theta$ generated 
by $\mathrm{NS}_\theta \cup \{\theta \setminus S\}$; in practice, we will typically 
think of $\mathrm{NS}_\theta \restriction S$ as the ideal of nonstationary subsets of $S$, considered as 
an ideal on $S$. If $X$ is a set and $\kappa$ is a cardinal, then $[X]^\kappa := \{y \subseteq X \mid 
|y| = \kappa\}$. If $X$ and $Y$ are two sets, then ${^Y}X$ denotes the set of all functions 
with domain $Y$ and codomain $X$.

To facilitate clean statements of hypotheses, we adopt the convention that 0 is not a 
limit ordinal.

\section{Canonical functions and the Galvin-Hajnal rank} \label{canonical_section}

Suppose that $S$ is an infinite set and $I$ is a proper ideal on $S$. As usual, we let 
$I^+$ denote the set of \emph{$I$-positive} subsets of $S$, i.e., $I^+ := 
\mc P(S) \setminus I$. Given two functions $\varphi ,
\psi \in {^S}\On$, we write $\varphi <_I \psi$ to denote the assertion that the set $\{i \in S \mid \psi(i) 
\leq \varphi(i)\}$ is in $I$. Define $=_I$, $\leq_i$, etc.\ in the obvious way. We will be particularly 
interested in the case in which $S$ is a stationary subset of a regular uncountable cardinal $\theta$ and 
$I = \mathrm{NS}_\theta \restriction S$, i.e., $I$ is the collection of nonstationary subsets of $S$. 
In this context, given two functions $\varphi,\psi \in {^S}\On$, we will write $\varphi <_S \psi$ instead 
of $\varphi <_{\mathrm{NS}_\theta \restriction S} \psi$ (and similarly with $\leq_S$, $=_S$, etc.). In particular, for functions 
$\varphi, \psi \in {^\theta}\On$, $\varphi <_\theta \psi$ will denote $\varphi <_{\mathrm{NS}_\theta} 
\psi$. Note that $\varphi <_S
\psi$ if and only if there is a club $C \subseteq \theta$ such that, for all $i \in C \cap S$, we have $
\varphi(i) < \psi(i)$.

Fix for the remainder of this section a regular uncountable cardinal $\theta$.
Given a stationary set $S \subseteq \theta$, the 
corresponding relation $<_S$ is well-founded and therefore has a rank function, which yields 
what is known as the \emph{Galvin-Hajnal rank}.

\begin{definition}[\cite{galvin_hajnal}] \label{galvin_hajnal_def}
	Suppose that $\theta$ is an uncountable regular cardinal and $S \subseteq \theta$ is 
	stationary. The \emph{Galvin-Hajnal rank} of a function $\varphi \in {^S}\On$, denoted 
	$\|\varphi\|_S$, is defined by recursion on $<_S$ by letting 
	\[
		\|\varphi\|_S := \sup\{\|\psi\|_S + 1 \mid \psi \in {^S}\On \text{ and } \psi <_S \varphi\}
	\]
	for all $\varphi \in {^S}\On$.
\end{definition}

It is readily verified by recursion on $\|\varphi\|_S$ that, for all stationary $T \subseteq S \subseteq 
\theta$ and all $\varphi \in {^S}\On$, we have $\|\varphi\|_S \leq \|\varphi \restriction T\|_T$. 
In general, it is quite possible to have strict inequality here. However, if $\varphi$ is what is 
known as a \emph{canonical} function, this inequality is in fact always an equality. With
this in mind, let us now recall the definition of and some basic facts about canonical functions.

By recursion on ordinals $\alpha$, attempt to define the \emph{canonical function on 
$\theta$ of rank $\alpha$}, $\varphi^\theta_\alpha \in {^\theta}\On$, as follows. If $\beta$ is 
an ordinal and $\varphi^\theta_\alpha$ has been defined for all $\alpha < \beta$, then let 
$\varphi^\theta_\beta$ be the least upper bound for $\langle \varphi^\theta_\alpha \mid 
\alpha < \beta \rangle$ with respect to $<_\theta$, if such a least upper bound exists. 
In other words, $\varphi^\theta_\beta \in {^\theta}\On$ is a function such that 
\begin{itemize}
	\item $\varphi^\theta_\beta$ is a $<_\theta$-upper bound for $\langle \varphi^\theta_\alpha \mid \alpha < 
	\beta \rangle$;
	\item if $\psi$ is another $<_\theta$-upper bound for $\langle \varphi^\theta_\alpha \mid \alpha < \beta 
	\rangle$, then $\varphi^\theta_\beta \leq_\theta \psi$.
\end{itemize}
If such a least upper bound does not exist, then $\varphi^\theta_\beta$ is undefined (and therefore 
$\varphi^\theta_\gamma$ is undefined for all $\gamma > \beta$ as well).

Note that $\varphi^\theta_\beta$ is not uniquely determined, but is unique up to $=_\theta$-equivalence. 
We will let $\Phi^\theta_\beta$ denote the set of all canonical functions on $\theta$ of 
rank $\beta$. We will slightly abuse notation and use $\varphi^\theta_\beta$ to denote an 
arbitrary element of $\Phi^\theta_\beta$. We will always be working in contexts that are invariant 
under $=_\theta$-equivalence, so this will not result in any loss of generality. The following well-known 
fact (see \cite[\S 1]{krueger_schimmerling} for an introduction to canonical functions of rank 
less than $\theta^+$) shows that, for all $\beta < \theta^+$, there are canonical functions on $\theta$ of 
rank $\beta$.

\begin{fact}
	Let $\beta < \theta^+$, and let $e:\theta \rightarrow \beta$ be a surjection. Then the 
	function $\varphi \in {^\theta}\theta$ defined by letting $f(i) = \otp(e``i)$ for all $i < \theta$ 
	is in $\Phi^\theta_\beta$. 
\end{fact}

The following proposition follows almost immediately from the definition of \emph{canonical 
	function}. 

\begin{proposition} \label{prop_12}
	Suppose that $\beta$ is an ordinal for which $\varphi^\theta_\beta$ is defined, 
	$\psi \in {^\theta}\On$, and the set $S := \{i < \theta \mid \psi(i) < \varphi^\theta_\beta(i)\}$ 
	is stationary in $\theta$. Then there is a stationary $S' \subseteq S$ and an $\alpha < \beta$ 
	such that $\psi(i) \leq \varphi^\theta_\alpha(i)$ for all $i \in S'$.
\end{proposition}

\begin{proof}
	Suppose not. Then, for all $\alpha < \beta$, there is a club $C_\alpha \subseteq \theta$ such 
	that $\varphi^\theta_\alpha(i) < \psi(i)$ for all $i \in S \cap C_\alpha$. Define a function 
	$\tau \in {^\theta}\On$ by letting 
	\[
		\tau(i) = \begin{cases}
			\psi(i) & \text{if } i \in S \\
			\varphi^\theta_\beta(i) & \text{if } i \in \theta \setminus S.
		\end{cases}
	\]
	Then, by our assumptions, $\tau$ is a $<_\theta$-upper bound for $\langle \varphi^\theta_\alpha \mid 
	\alpha < \beta \rangle$, so, by the definition of \emph{canonical function} we must have 
	$\varphi^\theta_\beta \leq_\theta \tau$, contradicting the fact that $S \subseteq \theta$ is stationary 
	and, for all $i \in S$, we have $\tau(i) = \psi(i) < \varphi^\theta_\beta(i)$.
\end{proof}

The following basic facts will be relevant to our arguments. Throughout the remainder 
of the paper, given a function $\varphi$ taking ordinal values, we let 
$\varphi + 1$ denote the function $\psi$ defined by setting $\dom(\psi) = \dom(\varphi)$ 
and $\psi(i) = \varphi(i)+1$ for all $i \in \dom(\varphi)$.

\begin{proposition} \label{successor_limit_prop}
  Suppose that $\beta > 0$ is an ordinal such that $\varphi^\theta_\beta$ is defined.
  \begin{enumerate}
    \item $\varphi^\theta_{\beta+1}$ is defined and $\varphi^\theta_{\beta + 1} =_\theta \varphi^\theta_\beta 
    +1$.
    \item If $\beta$ is a limit ordinal, then there is a club $C \subseteq \theta$ such that 
    $\varphi^\theta_\beta(i)$ is a limit ordinal for all $i \in C$.
  \end{enumerate}
\end{proposition}

\begin{proof}
  (1) Clearly, $\varphi^\theta_\beta + 1$ is a $<_\theta$-upper bound for $\langle \varphi^\theta_\alpha 
  \mid \alpha \leq \beta \rangle$. Moreover, if $\psi$ is any other $<_\theta$-upper bound, then there 
  must be a club $C \subseteq \theta$ such that $\varphi^\theta_\beta(i) + 1 \leq \psi(i)$ for all 
  $i \in C$, and therefore $\varphi^\theta_\beta + 1 \leq_\theta \psi$. It follows that 
  $\varphi^\theta_{\beta+1} =_\theta \varphi^\theta_\beta + 1$.
  
  (2) Suppose for sake of contradiction that $\beta$ is a limit ordinal and yet there is a stationary 
  set $S \subseteq \theta$ such that $\varphi^\theta_\beta(i) = \gamma_i + 1$ is a successor ordinal 
  for all $i \in S$. Define a function $\psi \in {^\theta}\On$ by letting
  \[
    \psi(i) = \begin{cases}
      \gamma_i & \text{if } i \in S \\
      \varphi^\theta_\beta(i) & \text{otherwise}
    \end{cases}
  \]
  for all $i < \theta$. By Proposition \ref{prop_12}, we can find a stationary $S' \subseteq S$ 
  and an ordinal $\alpha < \beta$ such that $\psi(i) \leq \varphi^\theta_\alpha(i)$ for all 
  $i \in S'$, and hence, by removing a nonstationary subset from $S'$ if necessary, we can assume 
  that $\psi(i) + 1 \leq \varphi^\theta_{\alpha+1}(i)$ for all $i \in S'$. But, by our definition of 
  $\psi$, we have $\psi(i) + 1 = \varphi^\theta_\beta(i)$ for all $i \in S'$, and hence 
  $\varphi^\theta_{\alpha+1} \not<_\theta \varphi^\theta_\beta$, contradicting the fact that $\alpha + 1 
  < \beta$.
\end{proof}

\begin{proposition} \label{canonical_rank_prop}
	Suppose that $\theta$ is a regular uncountable cardinal, $\beta$ is an ordinal such that 
	$\varphi^\theta_\beta$ is defined, and $S \subseteq \theta$ is stationary. Then 
	$\|\varphi^\theta_\beta \restriction S\|_S = \beta$.
\end{proposition}

\begin{proof}
	The proof is by induction on $\beta$, so we assume that, for all $\alpha < \beta$ and all 
	stationary $T \subseteq \theta$, we have 
	$\|\varphi^\theta_\alpha \restriction T \|_T = \alpha$. Since $\varphi^\theta_\alpha 
	<_\theta \varphi^\theta_\beta$ for all $\alpha < \beta$, it follows that 
	$\|\varphi^\theta_\beta \restriction S\|_S \geq \beta$.
	
	For the opposite inequality, fix a function $\psi \in {^S}\On$ with $\psi <_S \varphi^\theta_\beta 
	\restriction S$; it suffices to show that $\|\psi\|_S < \beta$. An application of Proposition 
	\ref{prop_12} yields a stationary $T \subseteq S$ and an $\alpha < \beta$ such that 
	$\psi(i) \leq \varphi^\theta_\alpha(i)$ for all $i \in T$. By the induction hypothesis, we have 
	$\|\varphi^\theta_\alpha \restriction T\|_T = \alpha$, so it follows that $\|\psi
	\restriction T\|_T \leq \alpha$. But then, since $T \subseteq S$, 
	this implies that $\|\psi\|_S \leq \alpha < \beta$, as desired.
\end{proof}

We now recall two results from \cite{jech_variation}, the first of which is already implicit in 
\cite{cardinal_arithmetic}.

\begin{theorem} {\cite[Corollary 2.3]{jech_variation}} \label{jech_thm_1}
  Suppose that $A$ is an infinite set, $I$ is an ideal on $A$, and $\langle \mu_a \mid a \in A \rangle$ 
  is a sequence of regular cardinals such that $\mu_a > |A|^+$ for all $a \in A$. Then there exist a 
  set $B \in I^+$, a regular cardinal $\lambda > |A|^+$, and a sequence 
  $\vec{f} = \langle f_\alpha \mid \alpha < \lambda \rangle$ such that $\vec{f}$ is $<_{I \restriction 
  B}$-increasing and $<_{I \restriction B}$-cofinal in $\prod_{a \in B} \mu_a$.
\end{theorem}

Our statement of the next theorem is less general than its statement in \cite{jech_variation}; we focus 
on ideals of the form $\mathrm{NS}_\theta \restriction S$ rather than the arbitrary normal ideals of 
\cite{jech_variation}.

\begin{theorem} {\cite[Main Theorem]{jech_variation}} \label{jech_thm_2}
  Suppose that
  \begin{enumerate}
    \item $\kappa$ is a singular cardinal and $\cf(\kappa) = \theta > \omega$;
    \item $\langle \kappa_i \mid i < \theta \rangle$ is an increasing, continuous sequence of cardinals 
    converging to $\kappa$;
    \item $\langle \mu_i \mid i < \theta \rangle$ is an increasing sequence of regular cardinals 
    such that, for some function $\varphi \in {^\theta}\On$, we have $\mu_i = \kappa_i^{+\varphi(i)}$ 
    for all $i < \theta$;
    \item $S \subseteq \theta$ is stationary;
    \item $\lambda$ is a regular cardinal and $\vec{f} = \langle f_\alpha \mid \alpha < \lambda \rangle$ 
    is a $<_S$-increasing and $<_S$-cofinal sequence from $\prod_{i \in S} \mu_i$.
  \end{enumerate}
  Then $\lambda \leq \kappa^{+\|\varphi \restriction S\|_S}$.
\end{theorem}

Putting these two results together yields the following corollary.

\begin{corollary} \label{scale_cor}
  Suppose that $\theta$ is a regular uncountable cardinal, $S \subseteq \theta$ is stationary, 
  and $\beta$ is an ordinal such that $\varphi^\theta_\beta$ is defined. Suppose also that 
  $\kappa$ is a singular cardinal, $\cf(\kappa) = \theta$, and $\langle \kappa_i \mid i < \theta 
  \rangle$ is an increasing, continuous sequence of cardinals converging to $\kappa$ 
  with $\kappa_0 > \theta$. Then there is 
  a stationary $S' \subseteq S$ and a sequence $\vec{f} = \langle f_\alpha \mid \alpha < 
  \lambda \rangle$ from $\prod_{i \in S'} \kappa_i^{+\varphi^\theta_\beta(i)+1}$ such that
  \begin{enumerate}
    \item $\vec{f}$ is $<_{S'}$-increasing and $<_{S'}$-cofinal in $\prod_{i \in S'} 
    \kappa_i^{+\varphi^\theta_\beta(i)+1}$;
    \item $\lambda \leq \kappa^{+\beta + 1}$.
  \end{enumerate}
  In particular, there exists a $<_{S'}$-cofinal subset $\mc F \subseteq \prod_{i \in S'} 
  \kappa_i^{+\varphi^\theta_\beta(i)+1}$ such that $|\mc F| \leq \kappa^{+\beta+1}$.
\end{corollary}

\begin{proof}
  By Proposition \ref{successor_limit_prop}, we have $\varphi^\theta_{\beta + 1} = \varphi^\theta_\beta 
  + 1$. For each $i \in S$, let $\mu_i := \kappa_i^{+\varphi^\theta_\beta(i)+1}$. Then, applying 
  Theorem \ref{jech_thm_1} to the set $S$, the ideal $\mathrm{NS}_\theta \restriction S$, and 
  the sequence $\langle \mu_i \mid i \in S \rangle$ of regular cardinals, we obtain a stationary 
  $S' \subseteq S$, a regular cardinal $\lambda > |A|^+$, and a sequence $\vec{f} = 
  \langle f_\alpha \mid \alpha < \lambda \rangle$ such that $\vec{f}$ is $<_{S'}$-increasing and 
  $<_{S'}$-cofinal in $\prod_{i \in S'} \mu_i$. Then Theorem \ref{jech_thm_2} implies that 
  $\lambda \leq \kappa^{+\|\varphi^\theta_{\beta + 1} \restriction S'\|_{S'}}$, so, by Proposition 
  \ref{canonical_rank_prop}, we have $\lambda \leq \kappa^{+\beta+1}$.
\end{proof}

\section{Density} \label{density_section}

In this section, we recall some notions of \emph{density} that will play a role throughout 
the paper. The first of these notions was the subject of Kojman's \cite{kojman_density}.

\begin{definition}\cite{kojman_density} \label{density_def}
  Suppose that $\theta \leq \mu$ are infinite cardinals. The \emph{$\theta$-density of 
  $\mu$}, denoted $d(\theta, \mu)$, is the minimal cardinality of a set 
  $\mc Y \subseteq [\mu]^\theta$ that is dense in $([\mu]^\theta, \subseteq)$, i.e., 
  for all $x \in [\mu]^\theta$, there is $y \in \mc Y$ such that $y \subseteq x$.
  
  If $\theta < \mu$, then $\underline{[\mu]^\theta}$ denotes the set of $x \in [\mu]^\theta$ such that 
  $\sup(x) < \mu$, and the \emph{lower $\theta$-density of $\mu$}, denoted 
  $\underline{d(\theta, \mu)}$, is the minimal cardinality of a set 
  $\mc Y \subseteq \underline{[\mu]^\theta}$ that is dense in 
  $(\underline{[\mu]^\theta}, \subseteq)$.
\end{definition}

\begin{remark}
  In \cite{kojman_density}, the $\theta$-density of $\mu$ is denoted by $\mc D(\mu, \theta)$. 
  We have chosen the notation $d(\theta, \mu)$ to match the established notation $m(\theta, \mu)$ for 
  \emph{meeting numbers}, which are among the cardinal characteristics considered here. Our notation 
  for both density numbers and meeting numbers follows \cite{matet_towers_and_clubs}.

  Note that, if $\theta < \mu$, then $\underline{d(\theta, \mu)} = \mu \cdot \sum_{\nu < \mu} 
  d(\theta, \nu)$. Therefore, if $\nu^\theta \leq \mu$ for all $\nu < \mu$, then 
  $\underline{d(\theta, \mu)} = \mu$.

  As remarked in \cite{kojman_density}, if $\cf(\mu) \neq \cf(\theta)$, then 
  \[
  d(\theta, \mu) = \underline{d(\theta, \mu)} = \mu \cdot \sum_{\nu < \mu} 
  d(\theta, \nu).
  \]
  In particular, if $\cf(\mu) \neq \cf(\theta)$ and $\nu^\theta \leq \mu$ for all $\nu < \mu$, 
  then $d(\theta, \mu) = \mu$.  
  
  If $\cf(\mu) = \cf(\theta)$, then a routine diagonalization argument shows that
  $d(\theta, \mu) \geq \mu^+$.
\end{remark}

The main result of \cite{kojman_density} is a version of Silver's theorem for the density 
number $d(\cf(\kappa), \kappa)$; this result served as direct motivation for the 
initial work that led to the results of this paper. Our main result here, when applied to 
the density number, will generalize and slightly improve upon the results of \cite{kojman_density}.

If $I$ is an ideal over a set $X$, then the \emph{density} of $I$, which we will 
denote $d(I)$, is the minimal cardinality of a set $\mc Y \subseteq I^+$ 
such that, for all $S \in I^+$, there is $T \in \mc Y$ such that $T \setminus S 
\in I$. We will particularly be interested in densities of the form $d(\mathrm{NS}_\theta 
\restriction S)$, where $S$ is a stationary set of a regular uncountable cardinal $\theta$. 
Concretely, $d(\mathrm{NS}_\theta \restriction S)$ is the minimal cardinality of a 
collection $\mc T$ of stationary subsets of $S$ such that, for every stationary 
$S' \subseteq S$, there is $T \in \mc T$ such that $T \setminus S'$ is nonstationary in $\theta$.

Finally, we introduce a notion of density that, in a sense, combines the two notions introduced 
in this section thus far.

\begin{definition} \label{stat_density_def}
  Suppose that $\theta$ and $\kappa$ are infinite cardinals, with $\theta$ regular, and 
  suppose that $S \subseteq \theta$ is stationary. Then 
  the \emph{stationarity density of ${^S}\kappa$}, which we denote by 
  $d_{\mathrm{stat}}({^S}\kappa)$, is the minimal cardinality of a family $\mc F$ of 
  functions such that
  \begin{enumerate}
    \item every $f \in \mc F$ is a function from a stationary subset of $S$ 
    to $\kappa$;
    \item for every function $g$ from a stationary subset of $S$ to $\kappa$, 
    there is $f \in \mc F$ such that the set 
    \[
      \{i \in \dom(f) \mid i \notin \dom(g) \text{ or } f(i) \neq g(i)\}
    \]
    is nonstationary. (Less precisely but more evocatively, $f$ is contained in 
    $g$ modulo a nonstationary set.)
  \end{enumerate}
  
  In analogy with lower density, we define the \emph{lower stationary density of 
  ${^S}\kappa$}, denoted $\underline{d_{\mathrm{stat}}({^S}\kappa)}$, in the 
  same way as $d_{\mathrm{stat}}({^S}\kappa)$, except that, in item (2), we only 
  consider functions whose ranges are bounded below $\kappa$ (and hence we can require 
  that all of our functions in $\mc F$ also have ranges bounded below $\kappa$).
\end{definition}

\begin{remark}
	Note that $d(\mathrm{NS}_\theta \restriction S) \leq \underline{d_{\mathrm{stat}}({^S}\kappa)}$. 
	Also, whenever $T \subseteq S$ are stationary subsets of $\theta$, we have
	\begin{itemize}	
	 \item $d(\mathrm{NS}_\theta \restriction T) \leq d(\mathrm{NS}_\theta 
	\restriction S)$;
	\item $d_{\mathrm{stat}}({^{T}}\kappa) \leq d_{\mathrm{stat}}({^S}\kappa)$;
	\item $\underline{d_{\mathrm{stat}}({^{T}}\kappa)} \leq 
	\underline{d_{\mathrm{stat}}({^S}\kappa)}$.
	\end{itemize} 
\end{remark}

\section{The general framework} \label{framework_section}

The cardinal characteristics we consider in this paper all have the following form:
a set $\mc X$ and a binary relation $\squig$ on $\mc X$ are fixed, and the relevant cardinal 
characteristic is then $\cf(\mc X, \squig)$, i.e., the minimal cardinality of a subset 
$\mc Y \subseteq \mc X$ such that, for all $x \in \mc X$, there is $y \in \mc Y$ such that 
$x \squig y$.

Because of the structural similarity of these cardinal characteristics, the inductive 
steps in the proofs of our main results end up being essentially the same, so in this section 
we prove a general lemma that we can directly apply to all of the specific situations 
under consideration here, and that we expect will find application beyond the scope of this 
paper, as well.

In order to state and prove our general lemma, let us fix some objects and notation for the 
remainder of this section:
\begin{itemize}
  \item $\kappa$ is a singular cardinal and $\cf(\kappa) = \theta > \omega$;
  \item $\langle \kappa_i \mid i < \theta \rangle$ is an increasing, continuous sequence of 
  cardinals converging to $\kappa$, with $\kappa_0 > \theta$;
  \item $\mc X$ is a set and $\squig$ is a binary relation on $\mc X$;
  \item for each $i < \theta$, $\mc X_i$ is a set and $\squig_i$ is a binary relation on $\mc X_i$;
  \item for each $i < \theta$, $\pi_i:\mc X \rightarrow \mc X_i$ is a function;
  \item $e:\mathrm{NS}_\theta^+ \rightarrow \mathrm{Card}$ is a function such that, for all 
  stationary $T \subseteq S \subseteq \theta$, we have $e(T) \leq e(S)$.
\end{itemize}

\begin{remark}
  To help orient the reader, let us preview here some of the eventual interpretations of these objects. 
  In a typical application, we might have $\mc X = \mc P(\kappa)$, $\mc X_i = \mc P(\kappa_i)$, and 
  $\pi_i(x) = x \cap \kappa_i$, or $\mc X = {^\kappa}\kappa$, $\mc X_i = {^{\kappa_i}} \kappa_i$, and 
  $\pi_i(x)$ is a modification of $x \restriction \kappa_i$ to ensure that it takes values in 
  $\kappa_i$. The function $e$ will typically (though not always) output one of the density
  numbers introduced in Section \ref{density_section}.
\end{remark}

In this context, if $S \subseteq \theta$ is stationary and $\beta$ is an ordinal for which 
the canonical function $\varphi^\theta_\beta$ is defined (recall the discussion of 
canonical functions following Definition \ref{galvin_hajnal_def}), 
then let $\Psi(S, \beta)$ denote the following assertion: 
\begin{quote}
  \textbf{If} $\mc Z$ and $\langle \mc Y_i \mid i \in S \rangle$ are such that 
  \begin{enumerate}
    \item $\mc Z \subseteq \mc X$ and $\mc Y_i \subseteq \mc X_i$ for all $i \in S$;
    \item $|\mc Y_i| \leq \kappa_i^{+\varphi^\theta_\beta(i)}$ for all $i \in S$;
    \item for all $z \in \mc Z$, there is a club $C \subseteq \theta$ such that, for all 
    $i \in C \cap S$, there is $y \in \mc Y_i$ for which $\pi_i(z) \squig_i y$;
  \end{enumerate}
  \textbf{then} there is $\mc Y \subseteq \mc X$ such that $|\mc Y| \leq \kappa^{+\beta} + 
  d(\mathrm{NS}_\theta \restriction S) + e(S)$ and, for all $z \in \mc Z$, there is $y \in \mc Y$ 
  for which $z \squig y$.
\end{quote}

Let $\Psi^*(S, \beta)$ be defined in the same way, except, in the conclusion, we only require 
$|\mc Y| \leq \kappa^{+\beta} + e(S)$.

\begin{lemma} \label{framework_lemma}
  Suppose that $S \subseteq \theta$ is stationary and $\Psi(T, 0)$ holds for all stationary 
  $T \subseteq S$. Then, for all ordinals $\beta$ for which the canonical function 
  $\varphi^\theta_\beta$ is defined, $\Psi(S, \beta)$ holds.
\end{lemma}

\begin{proof}
  The proof is by induction on $\beta$, simultaneously for all stationary $S \subseteq \theta$. 
  Thus, fix an ordinal $\beta$ for which $\varphi^\theta_\beta$ is defined and a stationary set 
  $S \subseteq \theta$. By the hypothesis of the lemma, we can assume that $\beta > 0$, and by 
  the inductive hypothesis, we can assume that $\Psi(T, \alpha)$ holds for all $\alpha < \beta$ 
  and all stationary $T \subseteq S$. Fix $\mc Z$ and $\langle \mc Y_i \mid i \in S \rangle$ 
  as in the hypothesis of $\Psi(S, \beta)$; we will find $\mc Y \subseteq \mc X$ as in its conclusion. 
  For each $i \in S$, enumerate $\mc Y_i$ as $\langle y_{i, \xi} \mid \xi < 
  \kappa_i^{+\varphi^\theta_\beta(i)} \rangle$ (with repetitions if $|\mc Y_i| < 
  \kappa_i^{+\varphi^\theta_\beta(i)}$).
  
  Suppose first that $\beta = \beta' + 1$ is a successor ordinal. By Proposition 
  \ref{successor_limit_prop}, we can assume that $\varphi^\theta_\beta = \varphi^\theta_{\beta'}+1$. 
  Then, by Corollary \ref{scale_cor}, we can fix a stationary $S' \subseteq S$ and a $<_{S'}$-cofinal 
  family $\mc F \subseteq \prod_{i \in S'} \kappa_i^{+\varphi^\theta_\beta(i)}$ such that 
  $|\mc F| \leq \kappa^{+\beta}$. For each $f \in \mc F$, let 
  $\mc Y_{i,f} := \{y_{i, \xi} \mid \xi < f(i)\}$, and note that $|\mc Y_{i,f}| \leq \kappa_i^{+
  \varphi^\theta_{\beta'}(i)}$. Let $\mc Z_f$ be the set of $z \in \mc Z$ 
  for which there is a club $C \subseteq \theta$ such that, for all $i \in C \cap S'$, there is 
  $y \in \mc Y_{i,f}$ such that $\pi_i(z) \squig_i y$. Recalling that $d(\mathrm{NS}_\theta 
  \restriction S') \leq d(\mathrm{NS}_\theta \restriction S)$ and $e(S') \leq e(S)$,
  apply $\Psi(S', \beta')$ to $\mc Z_f$ and 
  $\langle \mc Y_{i,f} \mid i \in S' \rangle$ to find $\mc Y_f \subseteq \mc X$ such that 
  $|\mc Y_f| \leq \kappa^{+\beta'} + d(\mathrm{NS}_\theta \restriction S) + e(S)$ and, 
  for all $z \in \mc Z_f$, there is $y \in \mc Y_f$ such that $z \squig y$.
  
  Let $\mc Y = \bigcup_{f \in \mc F} \mc Y_f$; we claim that $\mc Y$ is as desired. 
  It is evident that $\mc Y \subseteq \mc X$ and $|\mc Y| \leq 
  \kappa^{+\beta} + d(\mathrm{NS}_\theta \restriction S) + e(S)$, so it remains to 
  verify that, for all $z \in \mc Z$, there is $y \in \mc Y$ such that $z \squig y$. For this, 
  it suffices to show that $\mc Z \subseteq \bigcup_{f \in \mc F} \mc Z_f$. To this end, fix 
  $z \in \mc Z$. 
  By assumption, there is a club $C \subseteq \theta$ such that, for all $i \in C \cap S'$, 
  there is $\xi_i < \kappa_i^{+\varphi^\theta_\beta(i)}$ for which $\pi_i(z) \squig_i y_{i, \xi_i}$. 
  Define a function $g \in \prod_{i \in S'} \kappa_i^{+\varphi^\theta_\beta(i)}$ by letting 
  \[
    g(i) = \begin{cases}
      \xi_i & \text{if } i \in C \\ 
      0 & \text{otherwise}
    \end{cases}
  \]
  for all $i \in S'$. Since $\mc F$ is $<_{S'}$-cofinal in $\prod_{i \in S'} 
  \kappa_i^{+\varphi^\theta_\beta(i)}$, we can find $f \in \mc F$ such that 
  $g <_{S'} f$, i.e., there is a club $D \subseteq \theta$ such that, for all $i \in D \cap S'$, 
  we have $g(i) < f(i)$. But then, for all $i \in D \cap C \cap S'$, we have 
  $\xi_i < f(i)$ and $\pi_i(z) \squig_i y_{i, \xi_i}$, so $D \cap C$ witnesses that $z$ is in 
  $\mc Z_f$, and we are done.
  
  Finally, suppose that $\beta$ is a limit ordinal. By Proposition \ref{successor_limit_prop}, we can 
  assume that $\varphi^\theta_\beta(i)$ is a limit ordinal for all $i \in S$.
  Let $\mc T$ be a collection of stationary subsets of $S$ such that $|\mc T| = 
  d(\mathrm{NS}_\theta \restriction S)$ and, for every stationary $S' \subseteq S$, there is 
  $T \in \mc T$ such that $T \setminus S' \in \mathrm{NS}_\theta$. For each $\alpha < \beta$ 
  and each $i \in S$, let $\mc Y^\alpha_i := \{y_{i, \xi} \mid \xi < 
  \kappa_i^{+\varphi^\theta_\alpha(i)}\}$. For each $\alpha < \beta$ and each $T \in \mc T$, let 
  $\mc Z_{T, \alpha}$ be the set of all $z \in \mc Z$ for which there is a club $C \subseteq \theta$ 
  such that, for all $i \in C \cap T$, there is $y \in \mc Y^\alpha_i$ such that $\pi_i(z) \squig_i 
  y$. Apply $\Psi(T, \alpha)$ to $\mc Z_{T, \alpha}$ and $\langle \mc Y^\alpha_i \mid i \in T \rangle$ 
  to find $\mc Y_{T, \alpha} \subseteq \mc X$ such that $|\mc Y_{T, \alpha}| \leq \kappa^{+\alpha} + 
  d(\mathrm{NS}_\theta \restriction S) + e(S)$ and, for all $z \in \mc Z_{T, \alpha}$, there is $y 
  \in \mc Y_{T, \alpha}$ such that $z \squig y$.
  
  Let $\mc Y = \bigcup \{\mc Y_{T, \alpha} \mid T \in \mc T, ~ \alpha < \beta\}$; we claim that 
  $\mc Y$ is as desired. 
  As in the successor case, it suffices to verify that $\mc Z \subseteq \bigcup \{\mc Z_{T, \alpha} 
  \mid T \in \mc T, ~ \alpha < \beta\}$. To this end, fix $z \in \mc Z$. By hypothesis, we can find a 
  club $C \subseteq \theta$ such that, for all $i \in C \cap S$, there is $\xi_i < 
  \kappa_i^{+\varphi^\theta_\beta(i)}$ for which $\pi_i(z) \squig_i y_{i, \xi_i}$. For each such $i$, 
  use the fact that $\varphi^\theta_\beta(i)$ is a limit ordinal to find $\gamma_i < 
  \varphi^\theta_\beta(i)$ such that $\xi_i < \kappa_i^{+\gamma_i}$. Define a function 
  $\psi \in {^\theta}\On$ by letting 
  \[
  	\psi(i) = \begin{cases}
  	  \gamma_i & \text{if } i \in C \cap S \\ 
  	  \varphi^\theta_\beta(i) & \text{otherwise}
  	\end{cases}
  \]
  for all $i < \theta$. By Proposition \ref{prop_12}, we can find a stationary $S' \subseteq C \cap S$ 
  and an $\alpha < \beta$ such that $\gamma_i = \psi(i) \leq \varphi^\theta_\alpha(i)$ for all 
  $i \in S'$. We can subsequently find a $T \in \mc T$ and a club $D \subseteq \theta$ such that 
  $D \cap T \subseteq S'$. Then, for all $i \in D \cap T$, we have $y_{i, \xi_i} \in \mc Y^\alpha_i$ 
  and $\pi_i(z) \squig_i y_{i, \xi_i}$, so $D$ witnesses that $z$ is in $\mc Z_{T, \alpha}$, and we 
  are done.
\end{proof}

Notice that the only place in which the value of $d(\mathrm{NS}_\theta \restriction S)$ 
plays a role in the proof of Lemma \ref{framework_lemma} is in the case in which $\beta$ is 
a limit ordinal (in the successor case it only makes an appearance via the inductive hypothesis). 
Therefore, if $\beta < \omega$ and $\Psi^*(T, 0)$ holds for all stationary $T \subseteq \theta$, 
then we can do away with $d(\mathrm{NS}_\theta \restriction S)$ in the conclusion of the lemma.
More precisely, the proof of the successor case of Lemma \ref{framework_lemma} yields 
the following corollary.

\begin{corollary} \label{framework_cor}
  Suppose that $S \subseteq \theta$ is stationary and $\Psi^*(T, 0)$ holds for all stationary 
  $T \subseteq S$. Then $\Psi^*(S, n)$ holds for all $n < \omega$.
\end{corollary}

The translation from Lemma \ref{framework_lemma} and Corollary \ref{framework_cor} to our main 
results will happen via the following corollary.

\begin{corollary} \label{main_framework_cor}
  Suppose that 
  \begin{enumerate}
    \item $\beta$ is an ordinal for which $\varphi^\theta_\beta$ is defined;
    \item $S := \{i < \theta \mid \cf(X_i, \squig_i) \leq \kappa_i^{+\varphi^\theta_\beta(i)}\}$ is 
    stationary in $\theta$;
    \item $\Psi(T, 0)$ holds for all stationary $T \subseteq S$.
  \end{enumerate}
  Then $\cf(X, \squig) \leq \kappa^{+\beta} + d(\mathrm{NS}_\theta \restriction S) + e(S)$. 
  If, moreover, $\beta < \omega$ and $\Psi^*(T, 0)$ holds for every stationary $T \subseteq S$, 
  then $\cf(X, \squig) \leq \kappa^{+\beta} + e(S)$.
\end{corollary}

\begin{proof}
  By Lemma \ref{framework_lemma}, we know that $\Psi(S, \beta)$ holds. Let $\mc Z = \mc X$ and, for 
  each $i \in S$, let $\mc Y_i \subseteq \mc X_i$ be such that $|\mc Y_i| \leq \kappa_i^{+
  \varphi^\theta_\beta(i)}$ and $\mc Y_i$ is $\squig_i$-cofinal in $\mc X_i$. In particular, for 
  all $z \in \mc Z$ and all $i \in S$, there is $y \in \mc Y_i$ such that $\pi_i(z) \squig_i y$. 
  Therefore, applying $\Psi(S, \beta)$ to $\mc Z$ and $\langle \mc Y_i \mid i \in S \rangle$ yields 
  a set $\mc Y \subseteq \mc X$ such that $|\mc Y| \leq \kappa^{+\beta} + d(\mathrm{NS}_\theta 
  \restriction S) + e(S)$ and, for all $z \in Z$, there is $y \in \mc Y$ such that $z \squig y$, i.e., 
  $\mc Y$ is $\squig$-cofinal in $\mc X$.
  
  For the ``moreover" clause, if $\beta < \omega$ and $\Psi^*(T, 0)$ holds for every stationary 
  $T \subseteq S$, then Corollary \ref{framework_cor} implies that $\Psi^*(S, \beta)$ holds. 
  Applying $\Psi^*(S, \beta)$ to the $\mc Z$ and $\langle \mc Y_i \mid i \in S \rangle$ of the 
  previous paragraph then yields $\mc Y \subseteq \mc X$ such that $|\mc Y| \leq \kappa^{+\beta} + e(S)$ 
  and $\mc Y$ is $\squig$-cofinal in $\mc X$.
\end{proof}

\section{Specific instances} \label{instance_section}

We now turn to applications of the general framework introduced in the previous section to particular 
cardinal characteristics at singular cardinals. Let us fix for this entire section cardinals $\kappa$ 
and $\theta$ such that $\omega < \theta = \cf(\kappa) < \kappa$, as well as an increasing, continuous 
sequence of cardinals $\langle \kappa_i \mid i < \theta \rangle$ converging to $\kappa$, with 
$\kappa_0 > \theta$.

Each cardinal characteristic $\mathfrak{cc}_\kappa$ we consider will entail a choice of a set $\mc X$ 
and a binary relation $\squig$ on $\mc X$ such that $\mathfrak{cc}_\kappa = \cf(\mc X, \squig)$. The 
sets $\langle \mc X_i \mid i < \theta \rangle$ and relations $\langle \squig_i \mid i < \theta 
\rangle$ will be defined analogously, so that $\mathfrak{cc}_{\kappa_i} = \cf(\mc X_i, \squig_i)$ for 
each $i < \theta$. We will also have natural restriction operations $\pi_i:\mc X \rightarrow \mc X_i$ 
for each $i < \theta$. Together with an appropriate choice of function $e:
\mathrm{NS}_\theta^+ \rightarrow \mathrm{Card}$, these assignments give rise to instances 
 of the formulas $\Psi(S, \beta)$ and $\Psi^*(S, \beta)$ from 
the previous section. The primary work of this section will consist of proving that all of these instances of 
$\Psi^*(T, 0)$ hold; Corollary \ref{main_framework_cor}
will then directly yield our main results.

We will begin by introducing each of the cardinal characteristics we will be considering 
and specifying the appropriate assignments for $\mc X$, $\squig$, $\langle (\mc X_i, 
\squig_i, \pi_i) \mid i < \theta \rangle$, and $e$ for each characteristic. We will then 
prove a lemma indicating that, for all of these assignments, the corresponding instance 
of $\Psi^*(T, 0)$ holds for every stationary $T \subseteq \theta$.

We note that for some of the cardinal characteristics, we will only define $\mc X_i$, $\squig_i$, and 
$\pi_i$ for limit ordinals $i < \theta$. Since the clubs $C$ in the statement of 
$\Psi(S, \beta)$ can always be assumed to consist entirely of limit ordinals, this is sufficient for 
our purposes.

\subsection{Meeting numbers}

The first cardinal characteristics we consider are the \emph{meeting numbers}.

\begin{definition} [\cite{matet_towers_and_clubs}]
  Suppose that $\sigma \leq \lambda$ are infinite cardinals. Then the \emph{meeting number} 
  $m(\sigma, \lambda)$ is the minimal cardinality of a collection $\mc Y \subseteq [\lambda]^\sigma$ 
  such that, for all $x \in [\lambda]^\sigma$, there is $y \in \mc Y$ such that $|x \cap y| = \sigma$.
\end{definition}

The meeting number $m(\sigma, \lambda)$ is of special interest when $\cf(\sigma) = \cf(\lambda)$, 
in which case a routine diagonalization argument implies that $m(\sigma, \lambda) > \lambda$. A 
result of Matet indicates that Shelah's Strong Hypothesis, a statement in PCF 
theory, is equivalent to the statement that all such meeting numbers take their minimal possible value:

\begin{theorem}[Matet, {\cite[Theorem 1.1]{matet_meeting_numbers}}]
  The following are equivalent:
  \begin{enumerate}
    \item Shelah's Strong Hypothesis;
    \item for every singular cardinal $\lambda$ of countable cofinality, $m(\aleph_0, \lambda) 
    = \lambda^+$;
    \item for all infinite cardinals $\sigma < \lambda$, we have $m(\sigma, \lambda) = \lambda^+$ 
    if $\cf(\sigma) = \cf(\lambda)$ and $m(\sigma, \lambda) = \lambda$ if $\cf(\sigma) \neq 
    \cf(\lambda)$.
  \end{enumerate}
\end{theorem}

We now specify assignments to define a version of the formula $\Psi^*(T, 0)$ appropriate for 
the meeting number. Let $\mc X = [\kappa]^\theta$ and, for each limit ordinal $i < \theta$, let 
$\mc X_i = [\kappa_i]^{\cf(i)}$. Define relations $\squig$ on $\mc X$ and $\squig_i$ on $\mc X_i$ by 
letting $x \squig y$ iff $|x \cap y| = \theta$ and $x \squig_i y$ iff $|x_i \cap y_i| = \cf(i)$ 
for all limit ordinals $i < \theta$. It is then evident that $m(\theta, \kappa) = 
\cf(\mc X, \squig)$ and $m(\cf(i), \kappa_i) = m(\cf(\kappa_i), \kappa_i) = \cf(\mc X_i, \squig_i)$ 
for all limit $i < \theta$. 

For each limit ordinal $i < \theta$, define a map $\pi_i:\mc X \rightarrow \mc X_i$ as 
follows. For any $i < \theta$ and $x \in \mc X$, if $\sup(x \cap \kappa_i) = \kappa_i$, then let 
$\pi_i(x)$ be an arbitrary unbounded subset of $x \cap \kappa_i$ of order type $\cf(i)$. 
If $\sup(x \cap \kappa_i) < \kappa_i$, then simply let $\pi_i(x)$ be an arbitrary element of 
$\mc X_i$. Note that, for every $x \in \mc X$ that is unbounded in $\kappa$, the set 
$\{i < \theta \mid \sup(x \cap \kappa_i) = \kappa_i\}$ is a club in $\theta$.

Let $e: \mathrm{NS}_\theta^+ \rightarrow \mathrm{Card}$ be 
the constant function taking value $\sum_{j < \theta} m(\theta, \kappa_j)$.

\subsection{Density}

As mentioned above, a version of Silver's theorem for density is proven in \cite{kojman_density}. 
We include density here for completeness, since our results are slightly more general than those 
of \cite{kojman_density}. 

We are interested in particular in the number $d(\theta, \kappa)$.
The setup for density will be similar to that for the meeting number. Again let $\mc X = [\kappa]^\theta$ 
and, for each limit ordinal $i < \theta$, let $\mc X_i = [\kappa_i]^{\mathrm{cf}(i)}$. Define 
relations $\squig$ on $\mc X$ and $\squig_i$ on $\mc X_i$ by letting $x \squig y$ or $x \squig_i y$ 
iff $x \supseteq y$. Then $d(\theta, \kappa) = \cf(\mc X, \squig)$ and $d(\cf(\kappa_i), \kappa_i) 
= \cf(\mc X_i, \squig_i)$ for all limit $i < \theta$. 

Define maps $\pi_i:\mc X \rightarrow \mc X_i$ for limit ordinals $i < \theta$ exactly as in the 
case of the meeting number in the previous subsection. Let $e : \mathrm{NS}^+_\theta \rightarrow 
\mathrm{Card}$ be the constant function taking value 
$\underline{d(\theta, \kappa)} = \sum_{j < \theta} d(\theta, \kappa_j)$ 
(recall Definition \ref{density_def}).

\subsection{The reaping number}

\begin{definition}
  Let $\lambda$ be an infinite cardinal.
  \begin{enumerate}
    \item If $x,y \in [\lambda]^\lambda$, then we say that $x$ \emph{splits} $y$ 
    if $|y \cap x| = |y \setminus x| = \lambda$.
    \item A family $\mc Y \subseteq [\lambda]^\lambda$ is \emph{unreaped} if there is 
    no single $x \in [\lambda]^\lambda$ that splits every element of $\mc Y$.
    \item The reaping number $\mathfrak{r}_\lambda$ is the minimum cardinality of 
    an unreaped family in $[\lambda]^\lambda$.
  \end{enumerate}
\end{definition}

A standard diagonalization argument shows that $\mathfrak{r}_\lambda > \lambda$ for every 
infinite cardinal $\lambda$.

Let $\mc X = [\kappa]^\kappa$ and, for each $i < \theta$, let $\mc X_i = [\kappa_i]^{\kappa_i}$. 
Define a relation $\squig$ on $\mc X$ by letting $x \squig y$ iff $x$ does not split $y$, i.e., 
either $|y \cap x| < \kappa$ or $|y \setminus x| < \kappa$. Similarly, for each $i < \theta$, 
define $\squig_i$ on $\mc X_i$ by letting $x \squig_i y$ iff $x$ does not split $y$. Then it is 
evident that $\mathfrak{r}_\kappa = \cf(\mc X, \squig)$ and $\mathfrak{r}_{\kappa_i} = 
\cf(\mc X_i, \squig_i)$ for all $i < \theta$.

For each $i < \theta$, define a map $\pi_i : \mc X \rightarrow \mc X_i$ as follows. For all 
$x \in \mc X$, if $|x \cap \kappa_i| = \kappa_i$, then let $\pi_i(x) = x \cap \kappa_i$. Otherwise, 
let $\pi_i(x) = \kappa_i$. Note that, for all $x \in [\kappa]^\kappa$, the set of $i < \theta$ 
for which $|x \cap \kappa_i| = \kappa_i$, and hence for which $\pi_i(x) = x \cap \kappa_i$, is a 
club in $\theta$. Finally, as in the case of density, let $e : \mathrm{NS}_\theta^+ \rightarrow 
\mathrm{Card}$ be the constant function taking value $\underline{d(\theta, \kappa)}$.

\subsection{The dominating number}

\begin{definition}
  Suppose that $\lambda$ is an infinite cardinal.
  \begin{enumerate}
    \item If $f, g \in {^\lambda}\mathrm{On}$, then $f <^* g$ if and only if 
    $|\{\eta < \lambda \mid g(\eta) \leq f(\eta)\}| < \lambda$.
    \item The \emph{dominating number} $\mathfrak{d}_\lambda$ is the minimal 
    cardinality of a family $\mc F \subseteq {^\lambda}\lambda$ such that 
    for every $g \in {^\lambda}\lambda$, there is $f \in \mc F$ such that $g <^* f$.
    \item More generally, for any limit ordinal $\sigma$, $\mathfrak{d}_{\lambda, 
    \sigma}$ is the minimal cardinality of a family $\mc F \subseteq {^\lambda}\sigma$ 
    such that, for every $g \in {^\lambda}\sigma$, there is $f \in \mc F$ such that 
    $g <^* f$.
  \end{enumerate}
\end{definition}

Proofs of the following basic facts can be found in \cite{garti_shelah_generalized}.

\begin{proposition}[{\cite[Claim 3.2 and Lemma 3.1]{garti_shelah_generalized}}] \label{d_prop}
  \label{garti_shelah_prop}
  Suppose that $\lambda$ is an infintie cardinal.
  \begin{enumerate}
    \item $\mathfrak{d}_\lambda > \lambda$ and $\cf(\mathfrak{d}_\lambda) > \lambda$.
    \item $\mathfrak{d}_\lambda = \mathfrak{d}_{\lambda, \cf(\lambda)}$.
  \end{enumerate}
\end{proposition}

The proofs of \cite[Claim 3.2 and Lemma 3.1]{garti_shelah_generalized} can be 
routinely adapted to yield the following generalization of the preceding proposition.

\begin{fact} \label{d_fact}
  Suppose that $\lambda$ is an infinite cardinal and $\sigma$ is a limit ordinal. 
  \begin{enumerate}
    \item $\mathfrak{d}_{\lambda, \sigma} > \lambda$ and $\cf(\mathfrak{d}_{\lambda, \sigma})
    > \lambda$.
    \item $\mathfrak{d}_{\lambda, \sigma} = \mathfrak{d}_{\lambda, \cf(\sigma)}$.
  \end{enumerate}  
\end{fact}

Recently, Shelah proved that, if $\kappa$ is a singular strong limit cardinal, 
then $\mathfrak{d}_\kappa$ always attains its maximal possible value. More precisely, 
he proved the following theorem:

\begin{theorem}[{\cite[Claim 1.5(2)]{shelah_d}}]
  Suppose that $\kappa$ is a singular cardinal and $\mu^{\mathrm{cf}(\kappa)} < \kappa$ 
  for all $\mu < \kappa$. Then $\mathfrak{d}_\kappa = 2^\kappa$.
\end{theorem}

We now specify assignments to define a version of the formula $\Psi^*(T, 0)$ appropriate for 
the dominating number.
Let $\mc X = {^\kappa}\theta$ and, for each limit ordinal $i < \theta$, let $\mc X_i = {^{\kappa_i}}i$. 
Let $\squig$ and $\squig_i$ be the relations $<^*$ on $\mc X$ and 
$\mc X_i$, respectively. By Proposition \ref{d_prop} and Fact \ref{d_fact}, we have 
$\mf d_\kappa = \cf(\mc X, \squig)$ and, for all limit $i < \theta$, we have $\mf d_{\kappa_i} = 
\cf(\mc X_i, \squig_i)$. 

For each limit ordinal $i < \theta$, define a map $\pi_i : \mc X \rightarrow \mc X_i$ as follows. 
For all $x \in \mc X$ and all $\eta < \kappa_i$, let 
\[
  \pi_i(x)(\eta) = \begin{cases}
    x(\eta) & \text{if } x(\eta) < i \\ 
    0 & \text{otherwise}.
  \end{cases}
\]
Note that we do indeed have $\pi_i(x) \in {^{\kappa_i}}i = \mc X_i$, as desired. Finally, let 
$e: \mathrm{NS}^+_\theta \rightarrow \mathrm{Card}$ be defined by letting $e(S) = 
\underline{d_{\mathrm{stat}}({^S}\kappa)}$ for all stationary $S \subseteq \theta$ (recall this 
notation from Definition \ref{stat_density_def}).

\subsection{The general lemma}

We now show that, in all of the cases introduced in this section, the corresponding version 
of $\Psi^*(T, 0)$ holds for all stationary $T \subseteq \theta$. Note that we may assume 
that $\varphi^\theta_0(i) = 0$ for all $i < \theta$, so clause (2) in the definition 
of $\psi(T, 0)$ asserts that $|\mc Y_i| \leq \kappa_i$ for all $i \in T$.

\begin{lemma} \label{base_lemma}
  For any cardinal characteristic $\mf{cc} \in \{m(\theta, \kappa), \mc D(\kappa, \theta), \mf r_\kappa, 
  \mf d_\kappa\}$, the corresponding formula $\Psi^*(T, 0)$ holds for every stationary 
  $T \subseteq \theta$.
\end{lemma}

\begin{proof}
  We begin with some general preliminaries and then split into cases depending on the 
  cardinal characteristic under consideration.

  Fix a stationary set $T \subseteq \theta$; we may assume that every element of $T$ is 
  a limit ordinal. Fix assignments for $\mc X$, $\squig$, $\langle (\mc X_i, 
  \squig_i, \pi_i) \mid i < \theta \rangle$, and $e$ corresponding to one of the cardinal 
  characteristics introduced in this section. To verify $\Psi^*(T, 0)$, fix a set $\mc Z   
  \subseteq \mc X$ and, for each $i \in T$, a set $\mc Y_i \subseteq \mc X_i$ such that
  \begin{itemize}
    \item for all $i \in T$, $|\mc Y_i| \leq \kappa_i$;
    \item for all $z \in \mc Z$, there is a club $C \subseteq \theta$ such that, for all 
    $i \in C \cap T$, there is $y \in \mc Y_i$ for which $\pi_i(z) \squig_i y$.
  \end{itemize}
  For each $i \in T$, enumerate $\mc Y_i$ as $\langle y_{i, \xi} \mid \xi < \kappa_i \rangle$ 
  (with repetitions if $|\mc Y_i| < \kappa_i$). 
  We will find $\mc Y \subseteq \mc X$ such that $|\mc Y| \leq \kappa + e(T)$ and, for all $z \in 
  \mc Z$, there is $y \in \mc Y$ for which $z \squig y$. Our method for doing this will 
  depend on the precise cardinal characteristic.
  
  \textbf{Case 1: $m(\theta, \kappa)$.}
  Recall that in this case $e(T) = \sum_{j < \theta} m(\theta, \kappa_j)$. Therefore, for each 
  $j < \theta$, we can fix a family $\mc W_j \subseteq [T \times \kappa_j]^\theta$ such that 
  \begin{itemize}
    \item $|\mc W_j| \leq e(T)$;
    \item for all $u \in [T \times \kappa_j]^\theta$, there is $w \in \mc W_j$ such that 
    $|w \cap u| = \theta$.
  \end{itemize}
  Let $\mc W := \bigcup_{j < \theta} \mc W_j$. For each $w \in \mc W$, let 
  $y^*_w = \bigcup \{y_{i, \xi} \mid 
  (i, \xi) \in w\}$, and note that $y^*_w \in [\kappa]^{\leq \theta}$. 
  
  Also, for each $j < \theta$, let $\mc Y^*_j \subseteq [\kappa_j]^\theta$ be such that $|\mc Y^*_j| \leq 
  e(T)$ and, for all $x \in [\kappa_j]^\theta$, there is $y \in \mc Y^*_j$ such that $|y \cap x| = 
  \theta$. Finally, let 
  \[
  \mc Y := \bigcup_{j < \theta} \mc Y^*_j \cup \left (
  \{y^*_w \mid w \in \mc W_j\} \cap [\kappa]^\theta \right).
  \]
  
  We claim that $\mc Y$ is as desired. It is evident that $\mc Y \subseteq \mc X$ and $|\mc Y| 
  \leq \kappa + e(T)$. It remains to show that, for all $z \in \mc Z$, there is $y \in \mc Y$ 
  such that $z \squig y$, i.e., $|y \cap z| = \theta$. To this end, fix $z \in \mc Z$. If 
  there is $j < \theta$ such that $z \subseteq \kappa_j$, then there is $y \in \mc Y^*_j$ 
  such that $|y \cap z| = \theta$, and we are done. Therefore, we may assume that 
  $z$ is unbounded in $\kappa$. By assumption, 
  there is a club $C \subseteq \theta$ such that, for all $i \in C \cap T$, we have 
  \begin{itemize}
    \item $\sup(z \cap \kappa_i) = \kappa_i$, and hence $\pi_i(z)$ is an unbounded subset of 
    $z \cap \kappa_i$ of order type $\cf(i)$;
    \item there is $\xi_i < \kappa_i$ for which $\pi_i(z) \squig_i y_{i, \xi_i}$, i.e., 
    $|y_{i, \xi_i} \cap \pi_i(z)| = \cf(i)$.
  \end{itemize}
  Since each $i \in T$ is a limit ordinal, we can apply Fodor's lemma to find a fixed $j < \theta$ and a stationary 
  $T' \subseteq C \cap T$ such that $\xi_i < \kappa_j$ for all $i \in T'$. Let $u = \{(i, \xi_i) \mid 
  i \in T'\}$. Then $u \in [T \times \kappa_j]^\theta$, so, by our choice of $\mc W_j$, we can find $w \in \mc W_j$ such that $|w \cap u| = 
  \theta$. It follows that the set $T'' := \{i \in T' \mid (i, \xi_i) \in w\}$ is unbounded in $\theta$.
  
  Consider the set $y^*_w := \bigcup \{y_{i, \xi} \mid (i, \xi) \in w\}$. Note that, for all 
  $i \in T''$, we know that $\pi_i(z)$ is an unbounded subset of 
  $z \cap \kappa_i$ of order type $\cf(i)$, and we know that $|y_{i, \xi_i} \cap \pi_i(z)| = 
  \cf(i)$. It follows that $y_{i, \xi_i} \cap \pi_i(z)$ is an unbounded subset of $z \cap \kappa_i$, 
  and therefore $y^*_w \cap z$ is an unbounded subset of $z$. In particular, $y^*_w \in 
  \mc Y$ and $z \squig y^*_w$, as desired.
  
  \textbf{Case 2: $d(\theta, \kappa)$.} The argument in this case is very similar to that in the 
  previous case, so we suppress some details. Since $e(T) = \underline{d(\theta, \kappa)} = 
  \sum_{j < \theta} d(\theta, \kappa_j)$, we can fix, for each $j < \theta$, a family 
  $\mc W_j \subseteq [T \times \kappa_j]^\theta$ that is dense in $([T \times \kappa_j]^\theta, \subseteq)$ 
  and has cardinality at most $e(T)$. Let $\mc W := \bigcup_{j < \theta} \mc W_j$ and, for each 
  $w \in \mc W$, let $y^*_w := \bigcup \{y_{i, \xi} \mid (i, \xi) \in w\}$.
  
  Also, for each $j < \theta$, fix $\mc Y^*_j \subseteq [\kappa_j]^\theta$ that is dense in 
  $([\kappa_j]^\theta, \subseteq)$ and has cardinality at most $e(T)$. Let $\mc Y := \bigcup_{j < \theta} 
  \mc Y^*_j \cup \left (\{y^*_w \mid w \in \mc W_j\} \cap [\kappa]^\theta \right)$.
  
  It is evident that $\mc Y \subseteq \mc X$ and $|\mc Y| \leq \kappa + e(T)$. To verify that 
  $\mc Y$ is as desired, fix $z \in \mc Z$. We must find $y \in \mc Y$ such that $y \subseteq z$. 
  If there is $j < \theta$ such that $z \subseteq \kappa_j$, then there is $y \in \mc Y^*_j$ such 
  that $y \subseteq z$, and we are done. Thus, suppose that $z$ is unbounded in $\kappa$. Then 
  there is a club $C \subseteq \theta$ such that, for all $i \in C \cap T$, we have 
  \begin{itemize}
    \item $\sup(z \cap \kappa_i) = \kappa_i$;
    \item there is $\xi_i < \kappa_i$ for which $y_{i, \xi_i} \subseteq \pi_i(z)$.
  \end{itemize}
  We can again find a stationary $T' \subseteq C \cap T$ and a fixed $j < \theta$ such that 
  $\xi_i < \kappa_j$ for all $i \in T'$. Let $u = \{(i, \xi_i) \mid i \in T'\}$, and 
  find $w \in \mc W_j$ such that $w \subseteq u$. As before, the set $T'' := \{i \in 
  T' \mid (i, \xi_i) \in w\}$ is unbounded in $\theta$; moreover, $w$ is \emph{precisely} 
  $\{(i, \xi_i) \mid i \in T''\}$. Therefore, $y^*_w = \bigcup\{y_{i,{\xi_i}} \mid 
  i \in T''\} \subseteq z$. It also follows exactly as in the previous case that 
  $y^*_w$ is unbounded in $\kappa$ and is therefore an element of $\mc Y$.
  
  \textbf{Case 3: $\mf r_\kappa$.} In this case, $e(T) = \underline{d(\theta, \kappa)}$. 
  Therefore, for each $j < \theta$, we can fix as in the previous case a set 
  $\mc W_j \subseteq [T \times 
  \kappa_j]^\theta$ such that $\mc W_j$ is dense in $([T \times \kappa_j]^\theta, 
  \subseteq)$ and $|\mc W_j| \leq d(T)$. Let $\mc W = \bigcup_{j < \theta} \mc D_j$ and, 
  for each $w \in \mc D$, let 
  $y^*_w = \bigcup \{y_{i, \xi} \mid (i,\xi) \in w\}$. Finally, let $\mc Y = \{y^*_w \mid w \in 
  \mc D\} \cap [\kappa]^\kappa$.
  
  We claim that $\mc Y$ is as desired. It is evident that $\mc Y \subseteq 
  [\kappa]^\kappa$ and $|\mc Y| \leq e(T)$. It remains to verify that no element of 
  $\mc Z$ splits every element of $\mc Y$. To this end, fix $z \in \mc Z$.
  By assumption, we can find a club $C \subseteq \theta$ such that,
  for all $i \in C \cap T$, we can find $\xi_i < \kappa_i$ such that either 
  $|y_{i, \xi_i} \cap \pi_i(z)| < \kappa_i$ or $|y_{i, \xi_i} \setminus \pi_i(z)| < \kappa_i$. 
  We can also assume that, for all $i \in C$, we have $|z \cap \kappa_i| = \kappa_i$, 
  and hence $\pi_i(z) = z \cap \kappa_i$.
  
  Find a stationary $S_0 \subseteq C \cap T$ such that either
  \begin{enumerate}
    \item for all $i \in S_0$, $|y_{i, \xi_i} \cap \pi_i(z)| < \kappa_i$; \emph{or}
    \item for all $i \in S_0$, $|y_{i, \xi_i} \setminus \pi_i(z)| < \kappa_i$.
  \end{enumerate}
  Without loss of generality, assume that (1) holds (the proof is symmetric if (2) holds).
  For each limit ordinal $i \in S_0$, we can find $j(i) < i$ such that 
  $\max\{\xi_i, |y_{i, \xi_i} \cap \pi_i(z)|\} < \kappa_{j(i)}$. Since $S_0$ is a stationary 
  subset of $\theta$, we can therefore find a stationary $S_1 \subseteq S_0$ and a 
  fixed $j < \theta$ such that $j(i) = j$ for all $i \in S_1$.
  
  Let $u = \{(i, \xi_i) \mid i \in S_1\}$. Then $u \in [\theta \times \kappa_j]^\theta$, 
  so we can find $w \in \mc W_j \subseteq \mc W$ such that $w \subseteq u$. 
  Note that $\{i \in S_1 \mid (i, \xi_i) \in w\}$ must be unbounded in $\theta$, 
  so we have $|y^*_w| = \kappa$, and hence $y^*_w \in \mc Y$. Moreover,
  \[
  	y^*_w \cap z = \bigcup_{(i, \xi) \in w} (y_{i, \xi} \cap z) = 
  	\bigcup_{(i, \xi) \in w} (y_{i, \xi} \cap \pi_i(z)) \subseteq 
  	\bigcup_{i \in S_1} (y_{i, \xi_i} \cap \pi_i(z)).
  \]
  Since $|y_{i, \xi_i} \cap \pi_i(z)| < \kappa_j$ for all $i \in S_1$ and $\kappa_j > 
  \theta = |S_1|$, it follows that $|y^*_w \cap z| \leq \kappa_j < \kappa$, so $z$ 
  does not split $y^*_w$, i.e., we have $z \squig y^*_w$, as desired.
  
  \textbf{Case 4: $\mf d_\kappa$.} In this case, $e(T) = \underline{d_{\mathrm{stat}} }
  (^{T} \kappa)$. Therefore, we can fix a family $\mc H$ such that 
  \begin{itemize}
    \item $|\mc H| = e(T)$;
    \item every element of $\mc H$ is a function from a stationary subset of $T$ to 
    $\kappa$ whose range is bounded below $\kappa$;
    \item for every function $g$ from a stationary subset of $T$ to $\kappa$ such that 
    the range of $g$ is bounded below $\kappa$, there is a function $h \in \mc H$ 
    such that $\{i \in \dom(h) \mid i \notin \dom(g) \text{ or } h(i) \neq g(i)\}$ 
    is nonstationary in $\theta$.
  \end{itemize}
  Recall also that, for each $i \in T$ and each $\xi < \kappa_i$, we have 
  $y_{i, \xi} \in {^{\kappa_i}}i$.  
  
  For each $h \in \mc H$, define a function $y^*_h \in {^\kappa}\theta$ as follows. 
  Let $T' = \dom(h)$. Since the range of $h$ is bounded below $\kappa$, we know that, 
  for all sufficiently large $i \in T'$, we have $h(i) < \kappa_i$, and hence
  $y_{i, h(i)}$ is defined. Therefore, for all sufficiently large $i \in T'$ and all 
  $\eta < \kappa_i$, we have $y_{i, h(i)}(\eta) < i$. Therefore, by 
  Fodor's Lemma, for each $\eta < \kappa$, we can find a $j < \theta$ such that 
  \[
  	\{i \in T' \mid \eta < \kappa_i, ~ h(i) < \kappa_i \text{ and } y_{i, h(i)}(\eta) = j\}
  \]
  is stationary in $\theta$. Let $y^*_h(\eta)$ be the least such $j$.
  
  Let $\mc Y = \{y^*_h \mid h \in \mc H\}$. We claim that $\mc Y$ is as desired. It is evident 
  that $\mc Y \subseteq {^\kappa}\theta$ and $|\mc Y| \leq e(T)$. It remains to verify that, 
  for every $z \in \mc Z$, there is $y \in \mc Y$ such that $z <^* y$. To this end, fix 
  $z \in \mc Z$. By assumption, we can find a club $C \subseteq \theta$ such that, for 
  all $i \in C \cap T$, there is $\xi_i < \kappa_i$ for which $\pi_i(z) <^* y_{i, \xi_i}$. 
  For each $i \in C \cap T$, let 
  \[
    E_i = \{\eta < \kappa_i \mid y_{i, \xi_i}(\eta) \leq \pi_i(z)(\eta)\},
  \]
  and note that $|E_i| < \kappa_i$. By two applications of Fodor's lemma, we can find a
  $j < \theta$ and a stationary $T' \subseteq C \cap T$ such that $\max\{\xi_i, |E_i|\} 
  < \kappa_j$ for all $i \in T'$. Then the map from $T'$ to $\kappa$ defined by sending 
  each $i \in T'$ to $\xi_i$ has range bounded below $\kappa$, so we can find 
  $h \in \mc H$ and a club $D \subseteq \theta$ such that, letting $T'' = \dom(h)$, 
  the following statement holds: for all $i \in D \cap T''$, we have $i \in T'$ and 
  $h(i) = \xi_i$. 
  
  We claim that $z <^* y^*_h$. To see this, first let $E := \bigcup_{i \in D \cap T''} E_i$, 
  and note that $|E| \leq \theta \cdot \kappa_j < \kappa$. It therefore suffices to show that 
  $z(\eta) < y^*_h(\eta)$ for all $\eta \in \kappa \setminus E$. To this end, fix 
  $\eta \in \kappa \setminus E$. Fix $\ell < \theta$ such that $\eta < \kappa_\ell$ and 
  $z(\eta) < \ell$. Then, for all $i \in {D \cap T''} \setminus \ell$, we have 
  $\pi_i(z)(\eta) = z(\eta)$. Moreover, for all such $i$, we have $\eta \notin E_i$, 
  and hence $z(\eta) < y_{i, \xi_i}(\eta) = y_{i, h(i)}(\eta)$. Recall that $y^*_h(\eta)$ 
  was defined in such a way that there is a stationary set $T^* \subseteq T''$ such 
  that $y^*_h(\eta) = y_{i, h(i)}(\eta)$ for all $i \in T^*$. Since 
  $D$ is a club in $\theta$, we can fix $i^* \in (D \cap T^*) \setminus \ell$. 
  But then we have $z(\eta) < y_{i^*, h(i^*)}(\eta) = y^*_h(\eta)$, as desired.
\end{proof}

Combining the results of this and the previous section, we obtain the precise statement 
of our main result.

\begin{mainthm}
  Suppose that 
  \begin{itemize}
    \item $\kappa$ is a singular cardinal and $\cf(\kappa) = \theta > \omega$;
    \item $\langle \kappa_i \mid i < \theta \rangle$ is an increasing, continuous 
    sequence of cardinals converging to $\kappa$;
    \item $\beta$ is an ordinal for which $\varphi^\theta_\beta$ exists;
    \item $S \subseteq \theta$ is stationary;
    \item $\mf{cc}$ is one of the cardinal characteristics $m(\theta, \kappa)$, $d(\theta, \kappa)$,
    $\mf r_\kappa$, or $\mf d_\kappa$, and, for each $i < \theta$, $\mf{cc}_i$ 
    is the corresponding cardinal characteristic $m(\cf(\kappa_i), \kappa_i)$, $d(\cf(\kappa_i), \kappa_i)$,
    $\mf r_{\kappa_i}$, or $\mf d_{\kappa_i}$;
    \item for all $i \in S$, we have $\mf{cc}_i \leq \kappa_i^{+\varphi^\theta_\beta(i)}$.
  \end{itemize}
  Then: 
  \begin{enumerate}
    \item If $\mf{cc} = m(\theta, \kappa)$, then $\mf{cc} \leq \kappa^{+\beta} + 
    \sum_{j < \theta} m(\theta, \kappa_j) + d(\mathrm{NS}_\theta \restriction 
    S)$. Moreover, if $\beta < \omega$, then $\mf{cc} \leq \kappa^{+\beta} + 
    \sum_{j < \theta} m(\theta, \kappa_j)$.
    \item If $\mf{cc} = d(\theta, \kappa)$ or $\mf{cc} = \mf r_\kappa$, then $\mf{cc} \leq \kappa^{+\beta} + 
    \underline{d(\theta, \kappa)} + d(\mathrm{NS}_\theta \restriction S)$. Moreover, 
    if $\beta < \omega$, then $\mf{cc} \leq \kappa^{+\beta} + 
    \underline{d(\theta, \kappa)}$.
    \item If $\mf{cc} = \mf d_\kappa$, then $\mf{cc} \leq \kappa^{+\beta} + 
    \underline{d_{\mathrm{stat}}({^S}\kappa)}$.
  \end{enumerate}
\end{mainthm}

\begin{proof}
  This follows directly from Lemma \ref{base_lemma} and Corollary \ref{main_framework_cor}.
\end{proof}

\section{Open questions} \label{question_section}

Throughout this section, $\kappa$ will denote an arbitrary infinite cardinal.
We feel that the most prominent cardinal characteristic at singular cardinals that is not covered 
by our results here is the \emph{ultrafilter number}, a close relative of the reaping number.

\begin{definition}
  The ultrafilter number $\mathfrak{u}_\kappa$ is the minimal size of a base for a uniform ultrafilter 
  over $\kappa$. In other words, it is the minimal cardinal $\lambda$ for which there exists a 
  uniform ultrafilter $U$ over $\kappa$ and a family $\mathcal{X} \subseteq U$ such that $|\mathcal{X}| 
  = \lambda$ and, for all $Y \in U$, there is $X \in \mathcal{X}$ such that $|X \setminus Y| < \kappa$.
\end{definition}

It is provable that $\mathfrak{u}_\kappa > \kappa$, and the ultrafilter number at singular cardinals 
has been extensively studied (cf.\ \cite{garti_shelah_ultrafilter}, \cite{garti_shelah_generalized}, 
\cite{honzik_stejskalova}, among others). 

\begin{question}
  Does a version of our Main Theorem hold for the ultrafilter number?
\end{question}

We briefly mentioned the almost disjointness number in the Introduction; 
we feel that some interesting questions can be formulated around it. We first recall the relevant 
definitions.

\begin{definition}
  An \emph{almost disjoint family} over $\kappa$ is a family $\mathcal{A} \subseteq [\kappa]^\kappa$ 
  such that, for all distinct $A,B \in \mathcal{A}$, we have $|A \cap B| < \kappa$. Such a family 
  is a \emph{maximal almost disjoint family} (MAD family) over $\kappa$ if, moreover, there is no almost 
  disjoint family $\mathcal{B}$ over $\kappa$ with $\mathcal{A} \subsetneq \mathcal{B}$.
\end{definition}

There are trivial ways to form MAD families over $\kappa$ (as an extreme case, $\{\kappa\}$ is a MAD 
family over $\kappa$). The \emph{almost disjointness number} $\mathfrak{a}_\kappa$ is defined to be 
the minimal cardinality of a \emph{nontrivial} MAD family over $\kappa$. It remains to specify what 
\emph{nontriviality} means. The most natural solution seems to be to say that a MAD family is 
nontrivial if and only if its cardinality is greater than $\cf(\kappa)$ (this is the approach taken, 
for instance, in \cite{kojman_kubis_shelah}). Under this definition, it is not difficult to prove that 
$\mathfrak{a}_\kappa \leq \mathfrak{a}_{\cf(\kappa)}$. However, there always exist MAD families over $\kappa$ of cardinality strictly greater than $\kappa$, so one could also declare that a MAD family 
over $\kappa$ is nontrivial if and only if its cardinality is greater than $\kappa$. Let us denote 
the version of the almost disjointness number arising from this more stringent definition of nontriviality 
by $\mathfrak{a}^*_\kappa$. 

\begin{question}
  Does a version of our Main Theorem hold for $\mathfrak{a}^*_\kappa$?
\end{question}

The most immediate specific incarnation of this question would be the following:

\begin{question}
  Suppose that $\kappa$ is a singular strong limit cardinal of uncountable cofinality and there is a 
  stationary subset $S \subseteq \kappa$ consisting of singular cardinals such that, for all $\mu \in S$, 
  there is a MAD family over $\mu$ of cardinality $\mu^+$. Must there be a MAD family over $\kappa$ 
  of cardinality $\kappa^+$.
\end{question}

\begin{definition}
 A graph $G$ is \emph{universal for graphs of size $\kappa$} if, for every graph $H$ with at most 
 $\kappa$-many vertices, there is an induced subgraph of $G$ that is isomorphic to $H$. Let 
 $\mathfrak{ug}_\kappa$ denote the minimal number of vertices in a graph $G$ that is universal for 
 graphs of size $\kappa$.
\end{definition}

\begin{question}
 Does a version of our Main Theorem hold for $\mathfrak{ug}_\kappa$?
\end{question}

We are also interested in whether analogues of Silver's theorem hold for statements that are not 
naturally formulated as statements about cardinal characteristics but which are consequences of 
$2^\kappa = \kappa^+$. We record here some particularly prominent examples.

\begin{definition}
 The polarized partition relation $\begin{pmatrix} \kappa^{+} \\ \kappa \end{pmatrix} 
 \rightarrow \begin{pmatrix} \kappa^+ \\ \kappa \end{pmatrix}^{1,1}_2$
 is the assertion that, for every function $c:\kappa^+ \times \kappa \rightarrow 2$, there are sets 
 $A \in [\kappa^+]^{\kappa^+}$ and $B \in [\kappa]^{\kappa}$ such that $c \restriction A \times B$ is 
 constant. The negation of this relation is denoted by $\begin{pmatrix} \kappa^{+} \\ \kappa \end{pmatrix} 
 \not\rightarrow \begin{pmatrix} \kappa^+ \\ \kappa \end{pmatrix}^{1,1}_2$.
\end{definition}

Erd\H{o}s, Hajnal, and Rado prove in \cite{ehr} that, if $2^\kappa = \kappa^+$, then $\begin{pmatrix} 
\kappa^{+} \\ \kappa \end{pmatrix} \not\rightarrow \begin{pmatrix} \kappa^+ \\ \kappa \end{pmatrix}^{1,1}_2$. On the other hand, Garti and Shelah prove in \cite{garti_shelah_polarized} that, assuming the 
consistency of a supercompact cardinal, the positive relation $\begin{pmatrix} \kappa^{+} \\ 
\kappa \end{pmatrix} \rightarrow \begin{pmatrix} \kappa^+ \\ \kappa \end{pmatrix}^{1,1}_2$ consistently 
holds for a singular strong limit cardinal $\kappa$ (in their result, $\kappa$ can have either countable 
or uncountable cofinality).

\begin{question}
 Suppose that $\kappa$ is a singular cardinal of uncountable cofinality and there is a stationary set 
 $S \subseteq \kappa$ consisting of singular cardinals such that, for all $\mu \in S$, we have 
 $\begin{pmatrix} \mu^{+} \\ \mu \end{pmatrix} 
 \not\rightarrow \begin{pmatrix} \mu^+ \\ \mu \end{pmatrix}^{1,1}_2$. Must it 
 be the case that $\begin{pmatrix} \kappa^{+} \\ \kappa \end{pmatrix} 
 \not\rightarrow \begin{pmatrix} \kappa^+ \\ \kappa \end{pmatrix}^{1,1}_2$?
\end{question}

In an early draft of this paper, we included here a question about Aronszajn trees 
at double successors of singular cardinals. We then realized that existing work of 
Golshani and Mohammadpour \cite{golshani_mohammadpour} provides an answer to this 
question, so we give a very brief account of this here.

Recall that, for a regular uncountable cardinal $\lambda$, a 
\emph{$\lambda$-Aronszajn tree} is a tree of height $\lambda$ with no levels or 
branches of cardinality $\lambda$. If $\lambda = \mu^+$, then a $\lambda$-Aronszajn 
tree $T$ is \emph{special} if there is a function $f:T \rightarrow \lambda$ that is 
injective on chains. Note that a special $\mu^+$-Aronszajn tree remains special in 
any outer model in which $\mu^+$ is preserved.
By a result of Specker \cite{specker}, if 
$\mu$ is regular and $\mu^{<\mu} = \mu$, then there is a special $\mu^+$-Aronszajn tree. 
In particular, if $2^\kappa = \kappa^+$, then there is a $\kappa^{++}$-Aronszajn 
tree. Therefore, the nonexistence of Aronszajn trees at the double successor of a 
singular strong limit cardinal requires a failure of the Singular Cardinals 
Hypothesis. In an earlier draft of this paper, we asked whether the existence of 
$\kappa^{++}$-Aronszajn trees satisfies a version of Silver's theorem. Here, 
we give a consistent negative answer to this question that follows almost 
immediately from the work in \cite{golshani_mohammadpour}.

\begin{theorem} \label{silver_a_tree_thm}
  Suppose that $\kappa$ is supercompact and $\lambda > \kappa$ is measurable. 
  Then there is a forcing extension in which $\kappa$ is a singular cardinal of 
  uncountable cofinality, there are $\mu^{++}$-Aronszajn trees for all $\mu < 
  \kappa$, but there are no $\kappa^{++}$-Aronszajn trees.
\end{theorem}

\begin{proof}[Proof sketch]
  We can assume that $\mathsf{GCH}$ holds in $V$. Therefore, by the aforementioned result
  of Specker, there is a special $\mu^{++}$-Aronszajn tree for all $\mu$.
  By the techniques of \cite{laver}, we can 
  arrange so that the supercompactness of $\kappa$ is preserved after adding any number of 
  Cohen subsets to $\kappa$ by forcing with an Easton-support iteration of length 
  $\kappa$ with the property that, for all $\alpha < \kappa$, either 
  the $\alpha^{\mathrm{th}}$ iterand is forced to be trivial or $\alpha$ is inaccessible 
  and the $\alpha^{\mathrm{th}}$ iterand is forced to be of the form 
  $\mathrm{Add}(\alpha, \beta)$ for some $\beta < \kappa$. Moreover, this iteration can 
  be defined so that it preserves all cardinals. (More precisely, we can let 
  $f:\kappa \rightarrow V_\kappa$ be a Laver function and, for all $\alpha < \kappa$, 
  let the $\alpha^{\mathrm{th}}$ iterand be forced to be trivial unless 
  $\alpha$ is inaccessible, $f(\alpha)$ is a cardinal, and $f``\alpha \subseteq V_\alpha$, 
  in which case the $\alpha^{\mathrm{th}}$ iterand is forced to be $\mathrm{Add}(\alpha, f(\alpha))$.) 
  
  Let $V_1$ be the extension 
  of $V$ by this forcing iteration. Since $V$ and $V_1$ have the 
  same cardinals, it remains true in $V_1$ that there is a special 
  $\mu^{++}$-Aronszajn tree for all $\mu$. Moreover, in $V_1$ it is the case that 
  the supercompactness of $\kappa$ is preserved after adding any number of Cohen subsets 
  to $\kappa$.
 
  Let $\delta < \kappa$ be a regular uncountable cardinal. By the results of 
  \cite{golshani_mohammadpour} (in particular the results of Sections 4 and 5 of 
  that paper), there is in $V_1$ a forcing notion $\bb{R}$ with the following properties:
  \begin{itemize}
    \item $\bb{R}$ preserves all cardinals below $\kappa^+$;
    \item $V_1^{\bb{R}} \models \cf(\kappa) = \delta$;
    \item $V_1^{\bb{R}} \models 2^\kappa = \lambda = \kappa^{++}$;
    \item $V_1^{\bb{R}} \models ``\text{there are no } \kappa^{++}\text{-Aronszajn trees}"$.
  \end{itemize}
  Since $\bb{R}$ preserves all cardinals below $\kappa^+$, it remains true in 
  $V_1^{\bb{R}}$ that there is a special $\mu^{++}$-Aronszajn tree for all 
  $\mu < \kappa$. Therefore, $V_1^{\bb{R}}$ is the desired forcing extension.
\end{proof}

\bibliographystyle{plain}
\bibliography{bib}

\end{document}